\documentclass[]{amsart}
\usepackage[sorted-cites]{amsrefs}
\usepackage{amsfonts}
\usepackage{graphicx}
\usepackage{amscd}
\usepackage{amsmath}
\usepackage{amssymb}
\usepackage{xcolor}

\makeatletter
\@namedef{subjclassname@2010}{%
  \textup{2010} Mathematics Subject Classification}
\makeatother

\usepackage{mathrsfs}
\usepackage{amsbsy}

\newtheorem{corollary}{\bf Corollary}

\newtheorem{lemma}{\bf Lemma}

\newtheorem{remark}{Remark}

\newtheorem{theorem}{\bf Theorem}

\theoremstyle{definition}

\numberwithin{equation}{section}

\makeatletter
\@namedef{subjclassname@2020}{%
  \textup{2020} Mathematics Subject Classification}
\makeatother

\begin{document}

\title[four-dimensional compact Ricci solitons]{On Euler characteristic and Hitchin-Thorpe inequality for four-dimensional \\compact Ricci solitons}

\author{Xu Cheng}
\author{Ernani Ribeiro Jr}
\author{Detang Zhou}

\address[X. Cheng]{Instituto de Matem\'atica e Estat\'istica, Universidade Federal Fluminense - UFF, 24020-140, Niter\'oi - RJ, Brazil}
\email{xucheng@id.uff.br}

\address[E. Ribeiro Jr]{Departamento  de Matem\'atica, Universidade Federal do Cear\'a - UFC, Campus do Pici, 60455-760, Fortaleza - CE, Brazil}\email{ernani@mat.ufc.br}

\address[D. Zhou]{Instituto de Matem\'atica e Estat\'istica, Universidade Federal Fluminense - UFF, 24020-140, Niter\'oi - RJ, Brazil}
\email{zhoud@id.uff.br}

\thanks{X. Cheng and D. Zhou were partially supported by CNPq/Brazil and FAPERJ/Brazil.}

\thanks{E. Ribeiro was partially supported by CNPq/Brazil [Grant: 309663/2021-0], CAPES/Brazil and FUNCAP/Brazil [Grant: PS1-0186-00258.01.00/21].}


\date{\today}

\begin{abstract}
In this article, we investigate the geometry of $4$-dimensional compact gradient Ricci solitons. We prove that, under an upper bound condition on the range of the potential function,   a $4$-dimensional compact gradient Ricci soliton must satisfy the classical Hitchin-Thorpe inequality.  In addition, some volume estimates are also obtained.
\end{abstract}

\maketitle
\section{Introduction}

A complete Riemannian metric $g$ on an $n$-dimensional smooth  manifold $M^n$ is called a {\it gradient Ricci soliton} if there exists a smooth function $f$ on $M^n$ such that the Ricci tensor $\textrm{Ric}$ of the metric $g$ satisfies the equation
\begin{equation}
\label{maineq}
\textrm{Ric}+ \textrm{Hess}\,f=\lambda g
\end{equation} for some constant $\lambda\in\Bbb{R}.$ Here, $\textrm{Hess}\,f$ denotes the Hessian of $f.$  A gradient Ricci soliton (\ref{maineq})  is called {\it shrinking},
 {\it  steady} or {\it expanding} if the real number $\lambda$ is positive, zero or negative, respectively. Ricci solitons  are self-similar solutions of the Ricci
flows. Moreover, since they  also arise as the  singularity models of  the Ricci flows (see  \cite{Hamilton1}, \cite{Perelman2}), it is  very important in  understanding  them.

 It was proved by Perelman \cite{Perelman0} that every compact Ricci soliton is a gradient Ricci soliton  (also see its proof  in \cite{ELM}).  Moreover, it is known that on a compact manifold $M^n,$ a gradient steady or expanding Ricci soliton is necessarily an Einstein metric (see \cite{MR1249376}).  On the other hand,  for real dimension $4,$ the first example of a compact non-Einstein gradient shrinking Ricci soliton was constructed in the 90’s by Koiso \cite{Ko} and Cao \cite{Cao1} on the compact complex surface $\Bbb{CP}^2\sharp (- \Bbb{CP}^2),$  where $(-\Bbb{CP}^2)$ denotes the complex projective space with the opposite orientation. Therefore, compact   non-Einstein Ricci solitons must be shrinking.  In dimension $n=2$, Hamilton \cite{Hamilton2} showed that any $2$-dimensional compact gradient shrinking Ricci soliton is isometric to a quotient of the sphere $\Bbb{S}^2.$ For $n=3$, by the works of Ivey \cite{Ivey} and Perelman \cite{Perelman2}, it is known that any $3$-dimensional compact gradient shrinking Ricci soliton is a finite quotient of the round sphere $\Bbb{S}^3.$  Even the  non-compact  gradient shrinking Ricci soliton have been classified in two and three dimensions.

 Unlike the cases of dimensions  $2$ and $3$, the classification of  higher dimension gradient shrinking Ricci soliton is still incomplete.  For dimension $4,$  after the   aforementioned works of Koiso \cite{Ko} and Cao \cite{Cao1}, Wang and Zhu \cite{WZ1} later obtained a gradient K\"ahler-Ricci soliton on $\Bbb{CP}^2\sharp 2(- \Bbb{CP}^2).$ It  remains to be determined whether a compact non-Einstein gradient Ricci soliton  is necessarily  K\"ahler.   In
any dimension, it is known that a compact gradient shrinking Ricci soliton with constant scalar
curvature must be Einstein; see  \cite[Eminenti,  La Nave and Mantegazza]{ELM}.
Even a $4$-dimensional non-compact gradient shrinking Ricci soliton with constant scalar
curvature is  rigid, too. Indeed, recently in \cite{CZ2021} the first and third authors of the present paper proved that  a $4$-dimensional non-compact gradient shrinking Ricci soliton with constant scalar
curvature $S=2\lambda$ must be isometric to a finite quotient of $\mathbb{S}^2\times \mathbb{R}^2$. This result,  together with the previous results of Petersen and Wylie \cite{MR2507581}, and  Fernández-López and  García-Río \cite{FR},  confirms  that  a $4$-dimensional  complete non-compact gradient shrinking Ricci soliton with constant scalar curvature is   isometric to  the Gaussian shrinking soliton $\Bbb{R}^4$, a finite quotient of  $\Bbb{S}^{2}\times\Bbb{R}^{2}$, or a finite quotient of $\Bbb{S}^{3}\times\Bbb{R}$.  In recent
years, there has been much progress concerning the classification problem of  $4$-dimensional gradient shrinking Ricci solitons; see, e.g., \cite{Cao1,caoALM11,CaoChen,CRZ,CZ2021,CH,CWZ,KW,MW,MW2,Naber} and the references therein. 

It is interesting to study the topological character of the compact $4$-dimensional gradient shrinking Ricci solitons. It follows by the works of Derdzi\'{n}ski \cite{Derd} and Fern\'{a}ndez-L\'opez and Garc\'ia-R\'io \cite{FLGR0} that a compact $4$-dimensional gradient shrinking Ricci soliton  $M$ has finite fundamental group; see \cite{ELM} for an alternative proof.  Consequently, its first Betti number $b_1(M)=0$ and hence, its  Euler cha\-rac\-teristic $\chi(M)$ and   signature $\tau(M)$ satisfy the inequality  $\chi(M)>|\tau(M)|$
 (see (\ref{eqBt})). However, it is well known that for  a  compact $4$-dimensional Einstein manifold $M$, the Hitchin-Thorpe  inequality holds  (\cite{Thorpe}, \cite{Hitchin};  see also \cite[Theorem 6.35]{MR2371700}), that is, 
\begin{equation}
\label{HTeq}
\chi(M)\geq \frac{3}{2}|\tau (M)|.
\end{equation}
 This inequality provides a  topological obstruction for the existence of an Einstein metric on a given compact $4$-dimensional manifold. As gradient Ricci solitons
are natural generalizations of Einstein manifolds and  the nontrivial gradient shrinking Ricci solitons on $\Bbb{CP}^2\sharp (- \Bbb{CP}^2)$ and $\Bbb{CP}^2\sharp 2(- \Bbb{CP}^2)$  indeed satisfy the Hitchin-Thorpe inequality,  the following question was raised  (see \cite[Problem 6]{caoALM11}):

\begin{flushright}
\begin{minipage}[t]{4.37in}
 \emph{``Does the Hitchin-Thorpe inequality hold for compact $4$-dimensional gradient shrinking Ricci solitons?"} 
 \end{minipage}
\end{flushright}

In the last years,  some partial answers were obtained. In this context, the assumed conditions  under which the Hitchin-Thorpe inequality holds are the following, respectively.

 \cite[Ma]{MR3128968}:  the scalar curvature $S$ and the volume of $M$ satisfy  $\int_M S^2 dV_{g}\leq 24\lambda^2 \text{Vol}(M);$

 \cite[Fern\'{a}ndez-L\'opez and Garc\'ia-R\'io]{MR2672426}: some  upper
diameter bounds in terms of the Ricci curvature;

\cite[Tadano]{Tadano}: a lower bound on the diameter involving the maximum and minimum values of the scalar curvature on $M^4,$ namely, 
 $$\left(2+\frac{\sqrt{6}}{2}\pi\right)\frac{\sqrt{S_{max}-S_{min}}}{\lambda}\leq  \text{diam}(M),$$ where $S_{\max}$ and $S_{\min}$ denote the maximum and minimum values of the  scalar curvature $S$ on $M$,  respectively;

  \cite[Zhang and Zhang]{ZZ}: the  manifold has non-positive Yamabe invariant and admits long time solutions of the normalized Ricci flow equation with bounded scalar curvature.

  In the K\"ahler case,  it is known that  any compact  K\"ahler gradient Ricci soliton of real dimension $4$ with the natural orientation  satisfies the inequality $2\chi(M)+ 3\tau (M)\geq 0$   (see  \cite{MR3128968}; this result was generalized to K\"ahler almost Ricci solitons in  \cite{MR3275265}). 
 
In this paper, we consider the question mentioned earlier.  Without loss of generality, we assume that the gradient shrinking Ricci solitons satisfy the equation 
\begin{equation}
\label{maineq-1}
\textrm{Ric}+ \textrm{Hess}\,f=\frac12 g.
\end{equation} 
This normalization may be achieved  by a scaling of the metric. We first establish the following result.

\begin{theorem}
\label{euler}
Let $(M^{4},\,g, f)$ be a $4$-dimensional compact gradient shrinking Ricci soliton  satisfying \eqref{maineq-1}.
Then it holds that
\begin{eqnarray}
\label{eq14.5}
8\pi^2\chi(M)\geq \int_M |W|^2 dV_{g} +\dfrac{1}{24}Vol(M)(5-e^{f_{\max}-f_{\min}}),
\end{eqnarray} where  $f_{\min}$ and $f_{\max}$ stand for the minimum and maximum  of  the potential function $f$ on $M^4,$ respectively, $\textrm{Vol}(M)$ denotes the volume of $M^4$,  and $W$ is the Weyl tensor.

Moreover, equality holds if and only if $g$ is an Eins\-tein metric  (in this case, $f$ is constant).
\end{theorem}

As a consequence of Theorem \ref{euler} we obtain the following corollary.

\begin{corollary}
\label{proAlg}
Let $(M^{4},\,g, f)$ be a $4$-dimensional compact gradient shrinking Ricci soliton   satisfying \eqref{maineq-1}.
 If $f_{\max}-f_{\min}\leq\log 5,$ then the Hitchin-Thorpe inequality 
\begin{eqnarray}\label{HT-a}
\chi(M)\geq \frac{3}{2}|\tau (M)|
\end{eqnarray}
holds on $M$.
\end{corollary}

\begin{remark} \label{rema1}
The  conclusion in Corollary \ref{proAlg} also  holds if one  replaces the assumption  $f_{\max}-f_{\min}\leq\log 5$  by  $S_{\max}-S_{\min}\leq\log 5,$ where $S_{\max}$ and $S_{\min}$ denote the maximum and minimum  of the  scalar curvature $S$ on $M$,  respectively.   Indeed, by a choice of $f$, the  scalar curvature $S$ of a normalized gradient shrinking Ricci soliton may satisfy $$S+\vert\nabla f\vert^2=f\,\,\,\,\hbox{and}\,\,\,\,\,S>0.$$  At a point $p\in M$ where the function $f$ takes the maximum, $(\nabla f) (p)=0$.
Hence,  $S(p)=f(p)=f_{\max}\geq f=S+\vert\nabla f\vert^2\geq S$. Consequently, $S_{\max}=S(p)=f_{\max}$. Let $q$ be a point where the function $f$ takes the minimum. Then $f_{\max}-f_{\min}= S_{\max}-S(q)\leq S_{\max}-S_{\min}$. 
\end{remark}

\begin{remark}
We point out that, under the choice of $f$ as in Remark \ref{rema1}, a normalized compact $4$-dimensional gradient shrinking Ricci soliton $(M^{4},\,g, f)$ with $f_{\max}\leq 3$ must satisfy the Hitchin-Thorpe inequality. Indeed, since $S>0$ and $S+\Delta f=2,$ one obtains that $$\int_{M}S^2 dV_{g}\leq S_{\max}\int_{M}S dV_{g}=2\,S_{\max}Vol(M).$$ Therefore, taking into account that $S_{\max}=f_{\max}\leq 3,$ we have $\int_{M}S^2 dV_{g}\leq 6 Vol(M).$ So, it suffices to use the result obtained by Ma \cite{MR3128968}  to conclude the Hitchin-Thorpe inequality holds on $M^4.$

\end{remark}

Again, as an application of Theorem \ref{euler}, we deduce the following volume  upper bounds depending on the range of the potential function.

\begin{theorem}
\label{Thm2}
Let $(M^{4},\,g, f)$ be a $4$-dimensional compact gradient shrinking Ricci soliton  satisfying \eqref{maineq-1}.
 Then the following assertions hold:
 \begin{align}
Vol(M)\left(5-e^{f_{\max}-f_{\min}}\right) &\leq 384\pi^2.\label{v1}
\end{align}
  Equality  holds if and only if  $(M,g)$ is a sphere $\mathbb{S}^4$ with the radius $\sqrt{6} $. 
 \begin{align}
Vol(M)\left(5-e^{f_{\max}-f_{\min}}\right) &\leq \mathcal{Y}(M,[g])^2,\label{v2}
\end{align}
 where $\mathcal{Y}(M,[g])$ stands for the Yamabe invariant of $(M^4,\,g).$ 
Moreover, equality holds  if and only if $g$ is an Einstein metric.

\end{theorem}

\medskip

{\bf Acknowledgement.} We would like to thank professor Huai-Dong Cao
for his interest in this work and helpful comments on an earlier version of the
paper.

\section{Notations and Preliminaries}

In this section we review some basic facts and present some lemmas that will be
used for the establishment of the main results.
 Throughout this paper, we adopt the following convention for the curvatures:

$$\text{Rm}(X,Y)=\nabla^2_{Y,X}-\nabla^2_{X,Y}, \quad  Rm(X,Y,Z,W)=g(\text{Rm}(X,Y)Z,W),$$
$$ K(e_i,e_j)=Rm(e_i,e_j,e_i,e_j), \quad  \text{Ric}(X,Y)=\text{tr}\, \text{Rm}(X, \cdot, Y, \cdot),$$
$$  R_{ij}=\text{Ric}(e_i, e_j), \quad  S=\text{tr}\,\text{Ric}.$$  Besides, the Weyl tensor $W$ is defined by the following decomposition formula
\begin{eqnarray}
\label{weyl}
R_{ijkl}&=&W_{ijkl}+\frac{1}{n-2}\big(R_{ik}g_{jl}+R_{jl}g_{ik}-R_{il}g_{jk}-R_{jk}g_{il}\big) \nonumber\\
 &&-\frac{R}{(n-1)(n-2)}\big(g_{jl}g_{ik}-g_{il}g_{jk}\big),
\end{eqnarray} where $R_{ijkl}$ stands for the Riemann curvature tensor of $(M^n,\,g).$
\vskip 0.2cm

Now, let $(M, g, f)$ be an $n$-dimensional gradient shrinking Ricci soliton satisfying
\begin{align}\label{soliton-2}
\text{Ric}+\text{Hess}\,f=\frac{1}{2} g.
\end{align} Tracing the soliton equation \eqref{soliton-2} we get 
 \begin{align}\label{prop1}
 S+\Delta f=\frac{n}{2},
 \end{align} where $S$ denotes the scalar curvature of $M$.
 
 Moreover,  it is known that $S+|\nabla f|^2-f $  is constant (see \cite{Ha1}) and hence, by adding a constant to $f$ if necessary, we have the equation
\begin{align}\label{prop2}
S+|\nabla f|^2=f. 
\end{align} It   follows from (\ref{prop1}) and (\ref{prop2}) that 
\begin{equation}
\label{kh12a}
\Delta_{f}f=\frac{n}{2}-f,
\end{equation} where $\Delta_f\cdot:=\Delta\cdot-\nabla_{\nabla f}\cdot$ is the drifted Laplacian.

In the sequel we recall the  useful equations for the curvatures of a gradient shrinking Ricci soliton.   For their proofs, we refer the reader to \cite{PW,ELM}.

\begin{eqnarray}
\nabla _{l}R_{ijkl} &=&R_{ijkl}f_{l}=\nabla _{j}R_{ik}-\nabla _{i}R_{jk}, \label{eq1}
\\
\nabla _{j}R_{ij} &=&R_{ij}f_{j}=\frac{1}{2}\nabla _{i}S,  \label{eq2} \\
 \Delta _{f}\mathrm{Rm} &=&\mathrm{Rm}+\mathrm{Rm}\ast \mathrm{Rm},  \label{eq4}\\
 \Delta _{f}R_{ij} &=&R_{ij}-2R_{ikjl}R_{kl},  \label{eq6}\\
\Delta _{f}S &=&S-2\left\vert \mathrm{Ric}\right\vert ^{2} =\langle\mathrm{Ric},g-2\mathrm{Ric}\rangle \label{eq5}.
\end{eqnarray}
In this paper, we consider that $M$ is compact.  In \cite{chen}, Chen proved that every complete gradient shrinking Ricci soliton has positive scalar curvature  unless it is flat. Hence  $S>0$   when $M$ is compact.

In the rest of this section, we focus on dimension $n=4.$ It is known that, on a four-dimensional oriented Riemannian manifold $M^4,$ the bundle of $2$-forms $\Lambda^2$ can be invariantly decomposed as a direct sum 
\begin{equation}
\label{lk1}
\Lambda^2=\Lambda^{+}\oplus\Lambda^{-},
\end{equation}
 where $\Lambda^{\pm}$  is the ($\pm 1)$-eigenspace of the Hodge star operator $\ast$. This decomposition is conformally invariant. In particular, let $\{e_{i}\}_{i=1}^{4}$ be an oriented orthonormal basis of the tangent space at any fixed point $p\in M^4.$ Then it gives rise to bases of $\Lambda^{\pm}$
\begin{equation}
\label{lkm1}
\Big\{e^{1}\wedge e^{2}\pm e^{3}\wedge e^{4},\,e^{1}\wedge e^{3}\pm e^{4}\wedge e^{2},\,e^{1}\wedge e^{4}\pm e^{2}\wedge e^{3}\Big\},
\end{equation} where each bi-vector has length $\sqrt{2}.$ Moreover, the decomposition (\ref{lk1}) allows us to conclude that the Weyl tensor $W$ is an endomorphism of $\Lambda^2 $ such that 
\begin{equation}
\label{ewq}
W = W^+\oplus W^-,
\end{equation} where $W^\pm:\Lambda^\pm M\longrightarrow\Lambda^\pm M$ are called the {\it self-dual} part and {\it anti-self-dual} part of the Weyl tensor $W$, respectively.  Thereby, we may fix a point $p\in M^4$ and diagonalize $W^\pm$ such that $w_i^\pm,$ $1\le i \le 3,$ are their respective eigenvalues. Also, these eigenvalues satisfy
\begin{equation}
\label{eigenvalues} w_1^{\pm}\leq w_2^{\pm}\leq w_3^{\pm}\,\,\,\,\hbox{and}\,\,\,\,w_1^{\pm}+w_2^{\pm}+w_3^{\pm}
= 0.
\end{equation} Hence, the following inequality holds

\begin{equation}
\label{eqdet}
\det W^{+}\leq \frac{\sqrt{6}}{18}|W^{+}|^{3}.
\end{equation} Moreover, equality holds in (\ref{eqdet}) if and only if $w_{1}^{+}=w_{2}^{+}.$

By Poincar\'e duality, it follows that, for all compact oriented $4$-dimensional manifolds, the {\it Euler characteristic} and {\it signature} of $M^4$ are given by
$$\chi(M)=2-2b_{1}(M)+b_{2}(M)\,\,\,\,\,\,\hbox{and}\,\,\,\,\,\,\tau(M)=b_{+}(M)-b_{-}(M),$$ where $b_{1}(M)$ and $b_{2}(M)=b_{+}(M)+b_{-}(M)$ are the first and second Betti numbers of $M^4,$ respectively. It turns out that
\begin{equation}
\label{eqBt}
\chi(M)\geq |\tau(M)|-2b_{1}(M)+2.
\end{equation}

The curvature and topology of a compact $4$-dimensional manifold are connected via the classical Gauss-Bonnet-Chern formula
\begin{equation}
\label{eqCBC}
\chi(M)=\frac{1}{8\pi^2}\int_M\left(|W^+|^2+|W^-|^2+\frac{S^2}{24}-\frac{1}{2}|\mathring{Ric}|^2\right)dV_g
\end{equation}and the Hirzebrush's theorem 
\begin{equation}
\label{eqsig}
\tau(M)=\frac{1}{12\pi^{2}}\int_M\left(|W^+|^2-|W^-|^2\right)dV_{g},
\end{equation} where $\mathring{Ric}=Ric-\frac{S}{4}g;$ for more details, see \cite[Chapter 13]{MR2371700}. It is easy to check from (\ref{eqCBC}) and (\ref{eqsig}) that every compact $4$-dimensional Einstein manifold must satisfy the Hitchin-Thorpe inequality. 

 We recall some useful expressions for the Euler cha\-rac\-teristic $\chi(M)$  and give their proof for the sake of completeness.

\begin{lemma}\label{lemma2} Let $(M^{4},\,g, f)$ be a compact $4$-dimensional gradient shrinking Ricci soliton satisfying \eqref{soliton-2}. Then
\begin{align}
8\pi^{2} \chi(M)&=\int_{M}|W|^2 dV_{g} + \frac{1}{6}Vol(M) -\frac{1}{12}\int_{M}\langle \nabla S, \nabla f\rangle dV_{g}, \label{euler-a}\\
8\pi^{2} \chi(M)&= \int_{M}|W|^2 dV_{g} + \frac{1}{2}Vol(M) -\frac{1}{12}\int_{M}S^2dV_g. \label{euler-b}
\end{align}
\end{lemma} 

\begin{proof}
 It follows from (\ref{prop1}) that 
\begin{equation}\int_{M}S\,dV_{g}=2\,Vol(M).\label{eq10-1}
\end{equation}
 Hence, by the Cauchy inequality, we have
$$4\,Vol(M)^2=\left(\int_M S dV_{g}\right)^2\leq Vol(M) \int_M S^2\, dV_{g}.
$$Consequently,
 \begin{equation}\label{eq15}
 \int_M S^2\, dV_{g}\geq 4\, Vol(M).
\end{equation} Moreover,  the equality in \eqref{eq15} holds if and only if $S=2$ and in this case, $M^4$ must be Einstein (see, e.g., \cite{ELM}).
 At the same time, observe that 

\begin{eqnarray}
\label{plqw}
\int_{M}S^2\, dV_{g}&=&\int_{M}S\left(2-\Delta f\right) dV_{g}\nonumber\\ &=& 2\int_{M}S\, dV_{g} - \int_{M}S\Delta f\,dV_{g}\nonumber\\&=& 4\,Vol(M)+\int_{M}\langle \nabla S,\nabla f\rangle dV_{g}.
\end{eqnarray} Therefore, \eqref{eq15} and  (\ref{plqw}) give that $\int_{M}\langle \nabla S,\nabla f\rangle dV_{g}\geq 0.$
On the other hand,  integrating  (\ref{eq5}) and then using \eqref{eq10-1}, we get
\begin{eqnarray}
\label{phn1}
2\int_{M}|Ric|^2\,dV_{g}&=& \int_{M}\left( S - \Delta S + \langle \nabla S, \nabla f\rangle \right)\, dV_{g}\nonumber\\ &=& 2\,Vol(M) +\int_{M}\langle \nabla S,\nabla f\rangle dV_{g}. 
\end{eqnarray}
Substituting \eqref{plqw} and \eqref{phn1} into the Gauss-Bonnet-Chern formula (\ref{eqCBC}) yields
\begin{eqnarray}
\label{eqth}
8\pi^{2} \chi(M)&=&\int_{M} \left(|W|^2 + \frac{S^2}{24}-\frac{1}{2}|\mathring{Ric}|^2\right) dV_{g}\nonumber\\&=& \int_{M}\left(|W|^2 +\frac{S^2}{6}-\frac{1}{2}|Ric|^2\right) dV_{g}\nonumber\\&=&\int_{M}|W|^2 dV_{g} + \frac{1}{6}Vol(M) -\frac{1}{12}\int_{M}\langle \nabla S, \nabla f\rangle dV_{g},
\end{eqnarray}
which is \eqref{euler-a}. Finally, plugging \eqref{plqw} into \eqref{eqth} gives \eqref{euler-b}.

\end{proof}

\section{Proof of the Main Results}

For a gradient shrinking Ricci soliton  $(M^{n},\,g,\, f)$,  as in \cite{MR2732975} by Cao and Zhou, we consider the sub-level set of the potential function: $$D(t):=\{x\in M;\, f(x)<t\}.$$
In this section, we first discuss  the absolute continuity of the integral of a  bound function on $D(t).$

\subsection{Absolute continuity of a integral  on $D(t)$}

 Recently,  Colding and Minicozzi \cite[Lemma 1.1]{Cold} proved the properties of the  critical set and the  level sets  of a proper function $f$    satisfying  $\Delta_f f=\frac{n}{2}-f$ on the set $\{x\in M;\,f(x)\geq \frac{n}{2}\}$, where  $M$ is an $n$-dimensional Riemannian manifold. In this paper,  we need a version of this lemma for the whole manifold $M$.  More precisely,
\begin{lemma} \label{lem-1}
Let $(M^{n}, g)$ be an $n$-dimensional complete (not necessarily compact) Riemannian manifold. Suppose  that  $f$ is a proper and non-constant function satisfying  $\Delta_f f=\frac{n}{2}-f$ on $M$. Let $\mathcal{C}$ be the set of critical points of $f.$ Then the following assertions occur:
\begin{enumerate}
\item[(a)]  The critical set $\mathcal{C}$ in $M$ is locally contained in a smooth $(n-1)$-dimensional manifold.
\item[(b)] Each level set $\{ f(x)=c\}$ has  $n$-dimensional  Hausdorff measure $\mathcal{H}^{n}(\{f=c\})=0.$ 
\item[(c)] The regular set $\partial D(t)\setminus\mathcal{C}$ is dense in $\partial D(t).$
\end{enumerate}

\end{lemma}

\begin{proof}    From  (\ref{kh12a}),  $\mathcal{C}\cap \{f=\frac n2\}$ is the singular set of the eigenfunction and hence    has locally finite $(n-2)$-dimensional Hausdorff measure  (see Theorem 1.1  in \cite{HHL}). On the other hand,   Lemma 1.1 in \cite{Cold} asserts  that $\mathcal{C}$ in  $\{f(x)> \frac{n}{2}\}$  is locally contained in a smooth $(n-1)$-dimensional manifold. Hence, it  remains to prove the  conclusion in  (a)  for  $\mathcal{C}$   in  $\{f< \frac{n}{2}\}$. This may be done by noticing that   (\ref{kh12a}) implies     $\Delta f>0$ on $\mathcal{C}\cap\{f<\frac{n}{2}\}$  and following the  argument  in \cite{Cold} for $\mathcal{C}\cap\{f>\frac{n}{2}\}$ with the appropriate adaptation.  So  (a) is proved.

 For any  value $c$ satisfying $\{f = c\}\neq \emptyset$,   the  set $\{f = c\} \setminus \mathcal{C}$  is a countable union of $(n-1)$-manifolds. This property together with (a) gives (b).

Now, we confirm (c).  For $t> \frac n2$, it is the assertion (3) in  Lemma 1.1 in \cite{Cold}. For $t< \frac n2$, similar to the proof of  (a),  the assertion in this case follows by using  \eqref{kh12a} and make the corresponding modifications of the argument of  \cite{Cold}. In the case  $t=\frac n2$, the assertion follows from the properties that $\mathcal{C}\cap \{f=\frac n2\}$ has locally finite $(n-2)$-dimensional Hausdorff measure and the  set $\{f = \frac n2\} \setminus \mathcal{C}$   is  locally  a smooth $(n-1)$-manifold.  

Thus, the proof is finished. 
\end{proof}

We recall a result proved  in \cite{Cold} by Colding and Minicozzi. 

\begin{lemma} \cite[Lemma 1.3]{Cold}\label{lem-cm} Suppose that $b$  is a proper $C^n$ function and $\mathcal{H}^n(\{|\nabla b|=0\})=0$ in $\{b\geq r_0\}$ for some fixed $r_0$. If $g$  is a bounded function and $Q(r)=\int_{r_0<b<r}g$, then $Q$ is absolutely continuous and $Q'(r)=\int_{b=r}\frac{g}{|\nabla b|}$ a.e.

\end{lemma} 

Using  Lemmas \ref{lem-1} and \ref{lem-cm},    the following lemma may be established.

\begin{lemma} \label{lem-1a}
Let $(M^{n},\,g,\, f)$ be an $n$-dimensional complete (not necessarily compact) gradient shrinking Ricci soliton satisfying \eqref{soliton-2}, where $f$ is non-constant. Suppose that $h$ is a bounded measurable function.  Then we have

\begin{enumerate}
\item[(1)]  the set of the critical points of $f$  and each level set of $f$  satisfy $\mathcal{H}^n(\{|\nabla f|=0\})=0$ and   $\mathcal{H}^{n}(\{f=c\})=0$, respectively.
\item[(2)] $ F(t):=\int_{D(t)}hdV_g$ is absolutely continuous and $ F'(t)=\int_{f=t}\frac{g}{|\nabla f|}$ a.e.,
where $D(t)=\{x\in M;\, f(x)<t\}.$

\end{enumerate}

\end{lemma}

\begin{proof} First note that $f$ satisfies \eqref{kh12a}, that is, $$\Delta_f f=\frac{n}{2}-f.$$
Also in  \cite{MR2732975},  Cao and Zhou proved that,  when $M$ is non-compact,  there exists some number $r_0>0$  so that  for all $r(x)\geq r_{0}$, 
\begin{equation}
\label{eqfbeh}
\frac{1}{4}\Big(r(x)-c\Big)^{2}\leq f(x)\leq \frac{1}{4}\Big(r(x)+c\Big)^{2}.
\end{equation} 
This implies that $f$ is proper. So (1) follows from (a) and (b) in Lemma \ref{lem-1}.

Next, we deal with the second assertion.  Indeed, by using that $S>0$ and \eqref{prop2}, we know   $f> 0$.  The properness of $f$ together with $f>0$  implies that $f$ may take the positive minimum $f_{\min}$. This fact and (a) in Lemma \ref{lem-1} imply that 
$$\displaystyle F(t)=\int_{f_{\min}\leq f< t}hdV_g=\int_{f_{\min}<f< t}hdV_g.$$
Thus, (2) follows from (1) and  Lemma \ref{lem-cm}.

\end{proof}

Now we will present the proof of the main results.

\subsection{Proof of Theorem \ref{euler}}

\begin{proof}  If $f$ is constant, then $M^4$ is Einstein and  \eqref {euler-a} implies that  the equality in  (\ref{eq14.5}) holds. So we consider the case  that $f$ is non-constant.

Let $a$ and $b$ be the minimum and maximum of $f$ on $M^4,$ respectively. Then, by using (2) in Lemma \ref{lem-1a}, we obtain that
\begin{eqnarray}\label{eq10} 
\int_{D(s)}\langle \nabla S,\nabla f\rangle dV_{g}&=&\int_a^s\int_{\partial D(t)}\dfrac{\langle \nabla S,\nabla f\rangle}{\vert\nabla f\vert}d\sigma dt\nonumber\\
&=&\int_a^s\int_{\partial D(t)}\langle \nabla S,\nu\rangle d\sigma dt\nonumber\\
&=&\int_a^s\left(\int_{ D(t)}\Delta S\, dV_{g}\right) dt\nonumber\\
&=&\int_a^s\left(\int_{ D(t)}(S+\langle \nabla S,\nabla f\rangle-2\vert Ric\vert^2) dV_{g}\right) dt.
\end{eqnarray} 
 In the second equality of \eqref{eq10},  $\nu$ is the outward  unit normal vector of $\partial D(t)$ and  in the third and fourth equalities, we have used the divergence theorem and   (\ref{eq5}), respectively.\\

\noindent Define the functions $\Phi(s)$ and $\Psi(s)$ by
$$\Phi(s)=\int_a^s\left(\int_{ D(t)}\langle \nabla S,\nabla f\rangle dV_{g}\right) dt$$ and 
$$\Psi(s)=\int_a^s\left(\int_{ D(t)}(S-2\vert Ric\vert^2) dV_{g}\right) dt.$$ 
Hence,  \eqref{eq10} becomes
\begin{equation}\label{eq11}
\Phi'(s)=\Phi(s)+\Psi(s).
\end{equation} 
Differentiating (\ref{eq11}), we get that
\begin{equation}\label{eq11a}
\Phi''(s)=\Phi'(s)+\Psi'(s).
\end{equation} 
Noting $\Phi'(a)=0,$ by \eqref{eq11a},  we obtain that
\begin{align*}
\Phi'(s)&=e^s\int_a^s\Psi'(t)e^{-t}dt.
\end{align*}
Consequently,
\begin{align}
\Phi'(b)&=e^b\int_a^b\Psi'(t)e^{-t}dt\nonumber\\
&=e^b\int_a^b \left(\int_{ D(t)}(S-2\vert Ric\vert^2) dV_{g}\right) e^{-t}dt\nonumber\\
&\leq e^b\int_a^b \left(\int_{ D(t)}(S-\dfrac{1}{2}S^2) dV_{g}\right) e^{-t}dt\nonumber\\
&= e^b\int_a^b \left(\int_{ D(t)}\left(\dfrac{1}{2}-\dfrac{1}{2}(S-1)^2\right) dV_{g}\right) e^{-t}dt\nonumber\\
&\leq \dfrac{1}{2}e^b\int_a^b\,Vol( D(t))\, e^{-t}dt\nonumber\\
&\leq \dfrac{1}{2}\,Vol(M)\, (e^{b-a}-1).\label{ineq-phi}
\end{align} Since $\Phi'(b)=\int_{ M}\langle \nabla S,\nabla f\rangle dV_{g}$, we get
\begin{equation}
\label{mnlk}
\int_{ M}\langle \nabla S,\nabla f\rangle dV_{g}\leq \dfrac{1}{2}Vol( M) (e^{b-a}-1).
\end{equation} Finally, plugging (\ref{mnlk}) into (\ref{eqth}) we conclude
\begin{align}\label{eq14.5a}
8\pi^2 \chi(M)&=\int_{M}|W|^2 dV_{g}+\frac{1}{6}  Vol(M) -\frac{1}{12}\int_{M} \langle \nabla S,\nabla f\rangle dV_{g}\nonumber\\
&\geq  \int_M |W|^2 dV_{g} +\dfrac{1}{24}Vol(M) (5-e^{b-a}),
\end{align} 
 which  is \eqref{eq14.5}. Now consider the case of   the equality in  \eqref{eq14.5}. Suppose, by contrary, that   $f$ is not constant.  Then, $\text{Vol}(D(t))<\text{Vol}(M)$ for any $t<b$.  Therefore the strict inequalities in  \eqref{ineq-phi}  and thus  in  \eqref{eq14.5a} must hold. This is a contradiction. So $f$ must be constant and $M^4$ is Einstein.
 This finishes the proof of the theorem. 
 
\end{proof}

\subsection{Proof of Corollary \ref{proAlg}}

\begin{proof}
Using  (\ref{eq14.5}) and (\ref{eqsig}), we get
\begin{eqnarray}
4\pi^2 \left(2\chi(M)\pm 3\tau(M)\right)&\geq & 2\int_{M} |W^{\pm}|^{2} dV_{g}\nonumber\\&&+\frac{1}{24}Vol(M)\left(5-e^{f_{\max}-f_{\min}}\right). 
\end{eqnarray} Hence, the assumption  $f_{\max}-f_{\min}\leq\log 5$ implies that 
$$\chi(M)\geq \frac{3}{2}|\tau (M)|,$$ as asserted. 
\end{proof}

\subsection{Proof of Theorem \ref{Thm2}}
\begin{proof}
Since $M^4$ has positive scalar curvature, using a result on a compact oriented Riemannian $4$-manifold with  positive scalar curvature, which was proved  by Gursky in \cite{Gursky} (see \cite{MR2198792} also), we know that $M^4$ must satisfy 
\begin{equation}\label{eq-gursky}
8\pi^2\left(\chi(M)-2\right)\leq \int_{M}|W|^2\, dV_{g}.
\end{equation} Moreover, equality holds if and only if $M^4$ is conformally equivalent to a sphere $\Bbb{S}^4.$   By \eqref{eq-gursky} and  (\ref{eq14.5}), we get
\begin{align}\label{eq-vola}
\frac{1}{24}Vol(M)\left(5-e^{f_{\max}-f_{\min}}\right) \leq 16\pi^2,
\end{align} which proves  \eqref{v1}.

 If the equality in \eqref{eq-vola} holds, then $M$  must be Einstein and conformally  equivalent to a sphere $\Bbb{S}^4$. Hence, $f$ is constant,  $\mathring{Ric}=0$, $S=2$, and  $W=0$.  The decomposition of the curvature tensor implies that the sectional curvature of $M$ must be constant $\frac16$.  Thus $M$ must be a standard  sphere $\mathbb{S}^4$ with the radius $\sqrt{6} $.  In particular, such a sphere has volume $\text{Vol}(M)=96\pi^2$.

Next, we deal with the second assertion in the theorem, i.e., the estimate on the Yamabe invariant. Since the $L^2$-norm of the Weyl tensor is conformally invariant in dimension $4,$ one sees that $$\int_{M}\left(\frac{S^2}{24}-\frac{1}{2}|\mathring{Ric}|^2\right) dV_{g}$$ is conformally invariant as well.  Let $\mathcal{Y}(M,[g])^2$ be the Yamabe invariant associated to $(M^4,\,g). $ Then, given the Yamabe metric $\overline{g}\in [g],$ we obtain

\begin{eqnarray}
\label{yam}
\mathcal{Y}(M,[g])^2 &=& \frac{1}{{Vol_{\overline{g}}(M)}} \left(\int_{M}\overline{S} dV_{\overline{g}}\right)^2 \nonumber\\&=& \int_{M}\overline{S}^{2} dV_{\overline{g}}\nonumber\\&\geq & \int_{M}\left(\overline{S}^{2} -12|\mathring{\overline{Ric}}|^2 \right) dV_{\overline{g}}\nonumber\\&=& \int_{M}\left(S^{2} -12|\mathring{Ric}|^2 \right) dV_{g},
\end{eqnarray} where we have used the conformally invariance in the last equality. Moreover, equality holds if and only if $(M^4, \,g)$ is conformally Einstein. 

On the other hand,  it is easy to check from (\ref{plqw}) and (\ref{phn1}) that 
\begin{eqnarray}\label{eq-ric}
\int_{M}|\mathring{Ric}|^2 dV_{g}&=&\int_{M}\left(|Ric|^2 -\frac{S^2}{4}\right) dV_{g}\nonumber\\&=&\frac{1}{2}\int_{M}S^2 dV_{g}-Vol(M)-\int_{M}\frac{S^2}{4} dV_{g}\nonumber\\ &=&\frac{1}{4}\int_{M}S^2 dV_{g} - Vol(M). 
\end{eqnarray}  Plugging \eqref{eq-ric} into (\ref{yam}), one obtains that

\begin{equation}
\mathcal{Y}(M,[g])^2 \geq -2\int_{M}S^2 dV_{g}+ 12 Vol(M).
\end{equation} Hence, we may use (\ref{euler-b}) to get

\begin{equation}
8\pi^2 \chi(M)\leq \int_{M}|W|^2 dV_{g} + \frac{1}{24}\mathcal{Y}(M,[g])^2 
\end{equation} and then it suffices to use  (\ref{eq14.5}) to see that

\begin{equation}\label{eq-vol2a}
Vol(M)\left(5-e^{f_{\max}-f_{\min}}\right)\leq \mathcal{Y}(M,[g])^2, 
\end{equation} as asserted.

Finally, if the equality  in \eqref{eq-vol2a} holds, then the equality in \eqref{eq14.5} must hold. Hence,  Theorem \ref{euler} implies that  $(M,g)$ must be Einstein with $4\text{Vol}(M)=\mathcal{Y}(M,[g])^2$. To prove the inverse, one only needs to note that   the equalities  in \eqref{eq14.5} and  \eqref{yam} hold if $(M, g)$ is Einstein. 

 So, the proof is finished.

\end{proof}

\begin{remark} We point out that  \eqref{v1} can be alternatively obtained  by using  \eqref{v2} as follows.     For any compact $4$-dimensional manifold  $M^4,$ due to Aubin and Schoen, one has the following inequality $ \mathcal{Y}(M^4,[g])^2\leq  \mathcal{Y}(\mathbb{S}^4,[g_0])^2$, where $g_0$ denotes the  metric of standard sphere $\mathbb{S}^4$.  Moreover,  equality holds if and only if $(M, g)$ is  conformally equivalent to a sphere $\Bbb{S}^4.$  At the same time, we have   $ \mathcal{Y}(\mathbb{S}^4,[g_0])^2=384\pi^2$. Hence, \eqref{v1} holds. For the equality case,  the same argument as in the proof of Theorem  \ref{Thm2} can be applied.
\end{remark}


\begin{bibdiv}
\begin{biblist}
\bib{MR2371700}{book}{
   author={Besse, Arthur L.},
   title={Einstein manifolds},
   series={Classics in Mathematics},
   note={Reprint of the 1987 edition},
   publisher={Springer-Verlag, Berlin},
   date={2008},
   pages={xii+516},
   isbn={978-3-540-74120-6},
   review={\MR{2371700}},
}
	
	
\bib{MR3275265}{article}{
    AUTHOR = {Brasil, Aldir}, author={Costa, Ezio}, author={Ribeiro Jr, Ernani},
     TITLE = {Hitchin-{T}horpe inequality and {K}aehler metrics for compact
              almost {R}icci soliton},
   JOURNAL = {Ann. Mat. Pura Appl. (4)},
  FJOURNAL = {Annali di Matematica Pura ed Applicata. Series IV},
    VOLUME = {193},
      YEAR = {2014},
    NUMBER = {6},
     PAGES = {1851--1860},
      ISSN = {0373-3114},
   MRCLASS = {53C25 (53C20 53C21 53C65)},
  MRNUMBER = {3275265},
MRREVIEWER = {Andrew Bucki},
       DOI = {},
       URL = {},
}


\bib{Cao1}{article}{
   author={Cao, Huai-Dong},
   title={Existence of gradient K\"{a}hler-Ricci solitons},
   conference={
      title={Elliptic and parabolic methods in geometry},
      address={Minneapolis, MN},
      date={1994},
   },
   book={
      publisher={A K Peters, Wellesley, MA},
   },
   date={1996},
   pages={1--16},
   review={\MR{1417944}},
}

\bib{caoALM11}{article}{
   author={Cao, Huai-Dong},
   title={Recent progress on Ricci solitons},
   journal={Adv. Lect. Math.},
   volume={11},
   date={2010},
   number={2},
   pages={1--38},
   }
   
   
   
   \bib{CaoChen}{article}{
   author={Cao, Huai-Dong},
   author={Chen, Qiang},
   title={On Bach-flat gradient shrinking Ricci solitons},
   journal={Duke Math. J.},
   volume={162},
   date={2013},
   number={6},
   pages={1149--1169},
   issn={0012-7094},
   review={\MR{3053567}},
}
 
 \bib{CRZ}{article}{
   author={Cao, Huai-Dong},
   author={Ribeiro Jr, Ernani},
   author={Zhou, Detang}, 
   title={Four-dimensional complete gradient shrinking Ricci solitons}
   journal={Journal f\"ur die reine und angewandte Mathematik (Crelle's Journal)} 
   volume={2021},
   number={778},
   date={2021},
   pages={127-144},
   doi={10.1515/crelle-2021-0026},
   }

   



\bib{MR2732975}{article}{
   author={Cao, Huai-Dong},
   author={Zhou, Detang},
   title={On complete gradient shrinking Ricci solitons},
   journal={J. Differential Geom.},
   volume={85},
   date={2010},
   number={2},
   pages={175--185},
   issn={0022-040X},
   review={\MR{2732975}}
   }
   
   \bib{CH}{article}{
   author={Cao, Xiaodong},
   author={Tran, Hung},
   title={The Weyl tensor of gradient Ricci solitons},
   journal={Geom. Topol.},
   volume={20},
   date={2016},
   number={1},
   pages={389--436},
   issn={1465-3060},
   review={\MR{3470717}},
}
   
   
   
\bib{CWZ}{article}{
   author={Cao, Xiaodong},
   author={Wang, Biao},
   author={Zhang, Zhou},
   title={On locally conformally flat gradient shrinking Ricci solitons},
   journal={Commun. Contemp. Math.},
   volume={13},
   date={2011},
   number={2},
   pages={269--282},
   issn={0219-1997},
   review={\MR{2794486}},
}
   
   
   
   
   
 \bib{chen}{article}{
   author={Chen, Bing-Long},
   title={Strong uniqueness of the Ricci flow},
   journal={J. Differential Geom.},
   volume={82},
   date={2009},
   number={2},
   pages={363--382},
   issn={0022-040X},
   review={\MR{2520796}},
}
	  
 
\bib{CZ2021}{article}{
 author={Cheng, Xu}, author={Zhou, Detang}, 
 title={Rigidity of four-dimensional gradient shrinking Ricci soliton}, 
 journal={ArXiv:2105.10744 [math.DG]},
  date={2021},
  }
  


\bib{Cold}{article}{
author={Colding, Tobias}, 
author={Minicozzi, William}, 
title={Optimal growth bounds for eigenfunctions},
journal={ArXiv:2109.04998 [math.DG]},
date={2021},
}


 

\bib{Derd}{article}{
AUTHOR = {Derdzi\'{n}ski, Andrzej},
     TITLE = {A {M}yers-type theorem and compact {R}icci solitons},
   JOURNAL = {Proc. Amer. Math. Soc.},
  FJOURNAL = {Proceedings of the American Mathematical Society},
    VOLUME = {134},
      YEAR = {2006},
    NUMBER = {12},
     PAGES = {3645--3648},
      ISSN = {0002-9939},
   MRCLASS = {53C20 (53C21)},
  MRNUMBER = {2240678},
       DOI = {},
       URL = {},
}

\bib{ELM}{article}{
   author={Eminenti, Manolo},
   author={La Nave, Gabriele},
   author={Mantegazza, Carlo},
   title={Ricci solitons: the equation point of view},
   journal={Manuscripta Math.},
   volume={127},
   date={2008},
   number={3},
   pages={345--367},
   issn={0025-2611},
   review={\MR{2448435}},
}


		
\bib{MR2672426}{article}{
   author={Fern\'{a}ndez-L\'{o}pez, Manuel},
   author={Garc\'{\i}a-R\'{\i}o, Eduardo},
   title={Diameter bounds and Hitchin-Thorpe inequalities for compact Ricci
   solitons},
   journal={Q. J. Math.},
   volume={61},
   date={2010},
   number={3},
   pages={319--327},
   issn={0033-5606},
   review={\MR{2672426}},
   doi={},
}

\bib{FLGR0}{article}{
   AUTHOR = {Fern\'{a}ndez-L\'{o}pez, Manuel}, author={Garc\'{\i}a-R\'{\i}o, Eduardo},
     TITLE = {A remark on compact {R}icci solitons},
   JOURNAL = {Math. Ann.},
  FJOURNAL = {Mathematische Annalen},
    VOLUME = {340},
      YEAR = {2008},
    NUMBER = {4},
     PAGES = {893--896},
      ISSN = {0025-5831},
   MRCLASS = {53C25 (53C20)},
  MRNUMBER = {2372742},
MRREVIEWER = {Carlo Mantegazza},
       DOI = {},
       URL = {},
}




\bib{FR}{article}{
   author={Fern\'{a}ndez-L\'{o}pez, Manuel},
   author={Garc\'{\i}a-R\'{\i}o, Eduardo},
   title={On gradient Ricci solitons with constant scalar curvature},
   journal={Proc. Amer. Math. Soc.},
   volume={144},
   date={2016},
   number={1},
   pages={369--378},
   issn={0002-9939},
   review={\MR{3415603}},
   doi={},
}	



\bib{Gursky}{article}{
    AUTHOR = {Gursky, Matthew J.},
     TITLE = {Locally conformally flat four- and six-manifolds of positive
              scalar curvature and positive {E}uler characteristic},
   JOURNAL = {Indiana Univ. Math. J.},
  FJOURNAL = {Indiana University Mathematics Journal},
    VOLUME = {43},
      YEAR = {1994},
    NUMBER = {3},
     PAGES = {747--774},
      ISSN = {0022-2518},
   MRCLASS = {53C20 (53C21)},
  MRNUMBER = {1305946},
MRREVIEWER = {Emmanuel Hebey},
       DOI = {},
       URL = {},
}


\bib{Hamilton1}{article}{author={Hamilton, Richard S.},
   title={Three-manifolds with positive Ricci curvature}, journal={J. Differential Geom.},, volume={17}, date={1982}, number={2}, pages={255--306},   review={\MR{664497}},
   }




\bib{Ha1}{article}{
   author={Hamilton, Richard S.},
   title={The formation of singularities in the Ricci flow},
   conference={
      title={Surveys in differential geometry, Vol. II},
      address={Cambridge, MA},
      date={1993},
   },
   book={
      publisher={Int. Press, Cambridge, MA},
   },
   date={1995},
   pages={7--136},
   review={\MR{1375255}},
}

\bib{Hamilton2}{article}{
   author={Hamilton, Richard S.},
   title={The formation of singularities in the Ricci flow},
   conference={
      title={Surveys in differential geometry, Vol. II},
      address={Cambridge, MA},
      date={1993},
   },
   book={
      publisher={Int. Press, Cambridge, MA},
   },
   date={1995},
   pages={7--136},
   review={\MR{1375255}},
}



\bib{HHL}{article}{
    AUTHOR = {Han, Qing},
author={Hardt, Robert},
author={Lin, Fanghua},
     TITLE = {Geometric measure of singular sets of elliptic equations},
   JOURNAL = {Comm. Pure Appl. Math.},
  FJOURNAL = {Communications on Pure and Applied Mathematics},
    VOLUME = {51},
      YEAR = {1998},
    NUMBER = {11-12},
     PAGES = {1425--1443},
      ISSN = {0010-3640},
   MRCLASS = {35J15 (35D99 49Q20)},
  MRNUMBER = {1639155},

      
}
	

\bib{Hitchin}{article}{
    AUTHOR = {Hitchin, Nigel},
     TITLE = {Compact four-dimensional {E}instein manifolds},
   JOURNAL = {J. Differential Geometry},
  FJOURNAL = {Journal of Differential Geometry},
    VOLUME = {9},
      YEAR = {1974},
     PAGES = {435--441},
      ISSN = {0022-040X},
   MRCLASS = {53C25},
  MRNUMBER = {350657},
       URL = {http://projecteuclid.org/euclid.jdg/1214432419},
}


\bib{Thorpe}{article}{
AUTHOR = {Thorpe, John A.},
     TITLE = {Some remarks on the {G}auss-{B}onnet integral},
   JOURNAL = {J. Math. Mech.},
    VOLUME = {18},
      YEAR = {1969},
     PAGES = {779--786},
   MRCLASS = {53.72},
  MRNUMBER = {0256307},
}
	

\bib{MR1249376}{article}{
   author={Ivey, Thomas},
   title={Ricci solitons on compact three-manifolds},
   journal={Differential Geom. Appl.},
   volume={3},
   date={1993},
   number={4},
   pages={301--307},
   issn={0926-2245},
   review={\MR{1249376}},
   doi={10.1016/0926-2245(93)90008-O},
}

\bib{Ivey}{article}{
   author={Ivey, Thomas},
   title={New examples of complete Ricci solitons},
   journal={Proc. Amer. Math. Soc.},
   volume={122},
   date={1994},
   number={1},
   pages={241--245},
   issn={0002-9939},
   review={\MR{1207538}},
}


\bib{KW}{article}{
   author={Kotschwar, Brett},
   author={Wang, Lu},
   title={Rigidity of asymptotically conical shrinking gradient Ricci solitons},
   journal={J. Differential Geom.},
   volume={100},
   date={2015},
   number={1},
   pages={55--108},
   review={\MR{3326574}},
}


\bib{Ko}{article}{
   author={Koiso, Norihito},
   title={On rotationally symmetric Hamilton's equation for K\"{a}hler-Einstein
   metrics},
   conference={
      title={Recent topics in differential and analytic geometry},
   },
   book={
      series={Adv. Stud. Pure Math.},
      volume={18},
      publisher={Academic Press, Boston, MA},
   },
   date={1990},
   pages={327--337},
   review={\MR{1145263}},
   doi={10.2969/aspm/01810327},
}




\bib{MR3128968}{article}{
   author={Ma, Li},
   title={Remarks on compact shrinking Ricci solitons of dimension four},
   language={English, with English and French summaries},
   journal={C. R. Math. Acad. Sci. Paris},
   volume={351},
   date={2013},
   number={21-22},
   pages={817--823},
   issn={1631-073X},
   review={\MR{3128968}},
   doi={},
}




\bib{MW}{article}{
   author={Munteanu, Ovidiu},
   author={Wang, Jiaping},
   title={Geometry of shrinking Ricci solitons},
   journal={Compos. Math.},
   volume={151},
   date={2015},
  number={12},
   pages={2273--2300},
   issn={0010-437X},
   review={\MR{3433887}},
}



\bib{MW2}{article}{
   author={Munteanu, Ovidiu},
   author={Wang, Jiaping},
   title={Positively curved shrinking Ricci solitons are compact},
   journal={J. Differential Geom.},
   volume={106},
   date={2017},
   number={3},
   pages={499--505},
   issn={0022-040X},
   review={\MR{3680555}},
}




\bib{Naber}{article}{
   author={Naber, Aaron},
   title={Noncompact shrinking four solitons with nonnegative curvature},
   journal={J. Reine Angew. Math.},
   volume={645},
   date={2010},
   pages={125--153},
   issn={0075-4102},
   review={\MR{2673425}},
}


\bib{Perelman2}{article}{author={Perelman, Grisha}, title={Ricci flow with surgery on three manifolds}, journal={ArXiv:math.DG/0303109}, date={},}



\bib{Perelman0}{article}{author={Perelman, Grisha}, title={The entropy formula for the Ricci flow and its geometric applications}, journal={ArXiv:math/0211159 [math.DG]}, date={},}



\bib{MR2507581}{article}{
   author={Petersen, Peter},
   author={Wylie, William},
   title={Rigidity of gradient Ricci solitons},
   journal={Pacific J. Math.},
   volume={241},
   date={2009},
   number={2},
   pages={329--345},
   issn={0030-8730},
   review={\MR{2507581}},
   doi={10.2140/pjm.2009.241.329},
}

\bib{PW}{article}{
   author={Petersen, Peter},
   author={Wylie, William},
   title={On the classification of gradient Ricci solitons},
   journal={Geom. Topol.},
   volume={14},
   date={2010},
   number={4},
   pages={2277--2300},
   issn={1465-3060},
   review={\MR{2740647}},
   doi={10.2140/gt.2010.14.2277},
}

\bib{MR2198792}{article}{
   author={Seshadri, Harish},
   title={Weyl curvature and the Euler characteristic in dimension four},
   journal={Differential Geom. Appl.},
   volume={24},
   date={2006},
   number={2},
   pages={172--177},
   issn={0926-2245},
   review={\MR{2198792}},
   doi={10.1016/j.difgeo.2005.08.008},
}


\bib{Tadano}{article}{
AUTHOR = {Tadano, Homare},
     TITLE = {An upper diameter bound for compact {R}icci solitons with
              application to the {H}itchin-{T}horpe inequality. {II}},
   JOURNAL = {J. Math. Phys.},
  FJOURNAL = {Journal of Mathematical Physics},
    VOLUME = {59},
      YEAR = {2018},
    NUMBER = {4},
     PAGES = {043507, 3},
      ISSN = {0022-2488},
   MRCLASS = {53C25},
  MRNUMBER = {3787338},
MRREVIEWER = {Ramiro Augusto Lafuente},
       DOI = {},
       URL = {},
}


\bib{WZ1}{article}{
author={Wang, Xu-Jia},
author={Zhu, Xiaohua},
title={K\"ahler-Ricci solitons on toric manifolds with positive first Chern class},
journal={Advances in Mathematics},
volume={188},
number={1},
date={2004},
pages={87-103},
}




\bib{ZZ}{article}{
 AUTHOR = {Zhang, Yuguang}, aurtho={Zhang, Zhenlei},
     TITLE = {A note on the {H}itchin-{T}horpe inequality and {R}icci flow
              on 4-manifolds},
   JOURNAL = {Proc. Amer. Math. Soc.},
  FJOURNAL = {Proceedings of the American Mathematical Society},
    VOLUME = {140},
      YEAR = {2012},
    NUMBER = {5},
     PAGES = {1777--1783},
      ISSN = {0002-9939},
   MRCLASS = {53C20 (53C44)},
  MRNUMBER = {2869163},
MRREVIEWER = {Yu Zheng},
       DOI = {},
       URL = {},
}






\end{biblist}
\end{bibdiv}
\end{document}